\documentclass[10pt,a4paper,reqno]{amsart}

\usepackage{amsmath,amssymb,amsthm,amsfonts}
\usepackage{amsbsy}
\usepackage[utf8]{inputenc}

\usepackage[greek,english]{babel}
\usepackage{url}
\usepackage{wrapfig}
\usepackage{enumitem}
\usepackage{multicol}
\usepackage{moreenum}

\usepackage{tikz}

\usepackage{xcolor}

\definecolor{LinkColor}{rgb}{0,0,0} 
\usepackage[colorlinks=true,linkcolor=LinkColor,citecolor=LinkColor,urlcolor= LinkColor, naturalnames, hyperindex, pdfstartview=FitH, bookmarksnumbered, plainpages]{hyperref}

\newtheorem{theorem}{Theorem}[section]
\newtheorem{corollary}[theorem]{Corollary}
\newtheorem{lemma}[theorem]{Lemma}
\newtheorem{proposition}[theorem]{Proposition}

\theoremstyle{definition}
\newtheorem{definition}[theorem]{Definition}
\newtheorem{remark}[theorem]{Remark}

\newcommand{\cut}{\textsf{cut}}
\newcommand{\SL}{\operatorname{SL}}
\newcommand{\GL}{\operatorname{GL}}
\newcommand{\Irr}{\operatorname{Irr}}
\newcommand{\Aut}{\operatorname{Aut}}
\newcommand{\Syl}{\operatorname{Syl}}
\newcommand{\FF}{\mathbb{F}}
\newcommand{\Z}{\textup{Z}}
\newcommand{\U}{\textup{U}}
\newcommand{\C}{\textup{C}}
\newcommand{\N}{\textup{N}}

\newcommand{\QQ}{\mathbb{Q}}

\begin{document}

\title[Integral group rings with trivial central units]{Integral group rings of solvable groups with trivial central units}
\author{Andreas B\"achle}
\address{Vakgroep Wiskunde, Vrije Universiteit Brussel, Pleinlaan 2, 1050 Brussels, Belgium}
\email{\href{mailto:abachle@vub.ac.be}{abachle@vub.ac.be}}
\thanks{The author is a postdoctoral researcher of the FWO (Research Foundation Flanders)}
\subjclass[2010] {16U60, 20C15, 20E45} 
\keywords{integral group ring, central units, inverse semi-rational, solvable groups, prime spectrum, Frobenius groups}

\begin{abstract} The integral group ring $\mathbb{Z} G$ of a group $G$ has only trivial central units, if the only central units of $\mathbb{Z} G$ are $\pm z$ for $z$ in the center of $G$. We show that the order of a finite solvable group $G$ with this property, can only have $2$, $3$, $5$ and $7$ as prime divisors, by linking this to inverse semi-rational groups and extending one result on this class of groups. We also classify the Frobenius groups whose integral group rings have only trivial central units.
\end{abstract}

\maketitle

\section{Introduction}

By $G$ we always denote a finite group. The integral group ring of $G$ is denoted by $\mathbb{Z} G$ and its group of units by $\U(\mathbb{Z} G)$. Clearly, all elements of $\pm G$ are units in $\mathbb{Z} G$, the so-called \emph{trivial units}. Already G.~Higman determined in his PhD thesis in 1940 the groups such that all units in $\mathbb{Z}G$ are trivial: all of them are Dedekind groups of small exponent. If $\pi(G)$ denotes the \emph{prime spectrum of $G$}, i.e.\ the set of prime divisors of the order of $G$, then $\pi(G) \subseteq \{2, 3\}$ in all these cases and clearly all of these groups are solvable.

It remains a fundamental problem to describe generators for $\U(\mathbb{Z} G)$ if not all units are trivial. It is well-known that the units of reduced norm $1$ together with the central units generate a subgroup of finite index in $\U(\mathbb{Z} G)$. For many finite groups it has been shown by E.~Jespers, G.~Leal, J.~Ritter and S.~Sehgal that the so-called bicyclic units generate a subgroup of finite index in the subgroup of reduced norm $1$ elements. These together with the so-called Bass units then generate a subgroup of finite index in the central units, see e.g.\ the recent monograph of E.~Jespers and \'A.~del R\'io \cite{JdR1}, Chapter~11, for a detailed treatment. It is relevant to know when the Bass units can be omitted, in other words when all central units are trivial, i.e.\ belong to $\pm \Z(G)$. Ritter and Sehgal studied groups $G$ with this property \cite{RS}. They obtained a description of the centers of the Wedderburn components of the rational group algebra of such groups together with a group theoretic characterization in terms of conditions on conjugacy classes, see Proposition \ref{RitterSehgal} below. This class of groups has recently again been extensively studied by G.K.~Bakshi, S.~Maheshwary and I.B.S.~Passi \cite{BMP,Maheshwary}. We adopt their notation for these groups.

\begin{definition} A group $G$ is said to be a \emph{\cut\ group} (``central units trivial'') if $\mathbb{Z} G$ only contains trivial central units, that is, if $u \in \U(\mathbb{Z} G)$ is central, then $u = \pm z$ for some $z \in \Z(G)$, the center of $G$. \end{definition}

In \cite[Theorems 1 and 2]{Maheshwary} Maheshwary shows that $\pi(G) \subseteq \{2, 3, 5, 7\}$ for $G$ a \cut\ group of odd order or a solvable \cut\ group of even order such that all elements are of prime power order. We show that this is true for all solvable groups:

\begin{theorem}\label{ps_solv_cut} Let $G$ be a solvable \cut\ group. Then $\pi(G) \subseteq \{2, 3, 5, 7\}$.\end{theorem}

This is best possible in the sense that none of the primes can be removed as can be seen from the groups in Theorem \ref{Frobenius_cut}. Note that for non-solvable \cut\ groups such a restriction on the prime spectrum cannot exist as all Weyl groups, in particular all symmetric groups $S_n$, are \cut\ groups. There is also a group theoretical interpretation of and interest in \cut\ groups. By Proposition \ref{RitterSehgal}, a \cut\ group is nothing but an inverse semi-rational group as introduced by D.~Chillag and S.~Dolfi in \cite{CD}. And this is also the path we will take to prove Theorem \ref{ps_solv_cut}: we will improve one of the results of Chillag and Dolfi on the prime spectrum of such groups that are solvable.

\vspace{\baselineskip}

As a second result, we classify the Frobenius groups with only trivial central units in their integral group rings. 
\begin{theorem}\label{Frobenius_cut} Let $K$ be a Frobenius complement.
\begin{enumerate}
  \item\label{theorem_even_complement} 
If $|K|$ is even and $K$ is the complement of a \cut\ Frobenius group $G$, then  $G$ is isomorphic to a group in the series in \eqref{theo:C32bC2} -- \eqref{theo:C72dQ8C3} with $b, c, d \in \mathbb{Z}_{\geq 1}$ or $G$ is one of the groups in \eqref{theo:C52Q8} -- \eqref{theo:C72SL23}.
  \vspace{-\baselineskip}
  \begin{multicols}{2}
   \begin{enumerate}
    \item\label{theo:C32bC2} $C_3^{b} \rtimes C_2$
    \item\label{theo:C32bC4} $C_3^{2b} \rtimes C_4$
    \item\label{theo:C32bQ8} $C_3^{2b} \rtimes Q_8$
    \item\label{theo:C5cC4} $C_5^{c} \rtimes C_4$
    \item\label{theo:C7dC6} $C_7^{d} \rtimes C_6$
    \item\label{theo:C72dQ8C3} $C_7^{2d} \rtimes (Q_8 \times C_3)$
    \end{enumerate}\columnbreak
    \begin{enumerate}[label=(\greek*),ref=\theenumi\greek*]
    \item\label{theo:C52Q8} $C_5^{2} \rtimes Q_8$
    \item\label{theo:C52C3C4} $C_5^{2} \rtimes (C_3 \rtimes C_4)$
    \item\label{theo:C52SL23} $C_5^{2} \rtimes \SL(2,3)$
    \item\label{theo:C72SL23} $C_7^{2} \rtimes \SL(2,3)$
   \end{enumerate}
  \end{multicols}
  \vspace{-\baselineskip}
  \noindent Conversely, for each of the above structure descriptions there is a \cut\ Frobenius group of that form and it is unique up to isomorphism.

 \item\label{odd_complement} If $|K|$ is odd, then there is a \cut\ Frobenius group $G$ with complement $K$ and Frobenius kernel $F$ if and only if $K \simeq C_3$ and one of the following holds
   \begin{enumerate}
     \item $F$ is a \cut\ $2$-group admitting a fixed point free automorphism of order $3$.\\ In particular, $F$ has order $2^{a}$ for some $a \in 2\mathbb{Z}_{\geq 1}$ and is an extension of an abelian group of exponent a divisor of $4$ by an abelian group of exponent a divisor of $4$.
     \item\label{kernel7group} $F$ is an extension of an elementary abelian $7$-group by an elementary abelian $7$-group, $\exp F = 7$ and $F$ admits a fixed point free automorphism of order $3$ fixing each cyclic subgroup of $F$.
   \end{enumerate}
\end{enumerate}
\end{theorem}

Theorem \ref{ps_solv_cut} is proved in Section \ref{sect:solvable}, while Theorem \ref{Frobenius_cut} is proved in Section \ref{sect:Frobenius}. We give corollaries describing \cut\ Frobenius groups with abelian Frobenius kernel (Corollary \ref{cor:abelian_kernel}) and \cut\ Camina groups (Corollary \ref{cor:camina}), see the end of Section \ref{sect:Frobenius}.

The notation is mainly standard. $C_n$ denotes a cyclic group of order $n$, $Q_8$ the quaternion group of order $8$. $\FF_q$ denotes a finite field with $q$ elements. Euler's totient function is indicated by $\phi$. Standard references are \cite{Huppert1,ScottGT}.

\section{Preliminaries}

We need the following group theoretical notion of (inverse) semi-rational groups that was coined by Chillag and Dolfi \cite{CD}. By $\sim$ we denote conjugacy (in a group that should always be clear from the context).

\begin{definition} Let $G$ be a group and $x \in G$. 
\begin{enumerate}
 \item $x$ is \emph{rational in $G$} if $x^j\sim x$ for all $j \in \mathbb{Z}$ with $(j, o(x)) = 1$.
 \item The element $x$ is called \emph{inverse semi-rational in $G$} if $x^j$ is conjugate to $x$ or $x^{-1}$ for all $j \in \mathbb{Z}$ with $(j, o(x)) = 1$.
 \item $x$ is \emph{semi-rational in $G$} if each if $x^j$ is conjugate to $x$ or $x^m$ for all $j \in \mathbb{Z}$ with $(j, o(x)) = 1$ for some fixed $m$.
\end{enumerate}
A group is called \emph{rational} (respectively \emph{inverse semi-rational}, \emph{semi-rational}) if every of its elements is respectively of that type.  \end{definition}
 
For a character $\psi$ (ordinary or Brauer), $\QQ(\psi)$ denotes the field of character values of $\psi$, that is, the field extension of $\QQ$ generated by $\{\psi(g) \colon g \in G\}$ (respectively by $\{\psi(g) \colon g \in G_{p'}\}$ if $\psi$ is a $p$-Brauer character and $G_{p'}$ denotes the elements of $G$ with an order not divisible by $p$). It is well-known that a group $G$ is rational if and only if  for each $\chi \in \Irr(G)$, $\QQ(\chi) = \QQ$, whence the name.
 
 The following characterization of \cut\ groups is essentially due to Ritter and Sehgal. We link it to inverse semi-rational groups.

\begin{proposition}\label{RitterSehgal} Let $G$ be a group. The following are equivalent.
\begin{enumerate}
 \item\label{prop:cut} $G$ is a \cut\ group.
 \item\label{prop:isr} $G$ is inverse semi-rational.
 \item\label{prop:conjugategenerators} For every $x \in G$ and $k \in \mathbb{Z}$ with $(k, |G|) = 1$, $x^k \sim x$ or $x^k \sim x^{-1}$.
 \item\label{prop:wedderburndecomp} If $\QQ G \simeq \bigoplus_{k = 1}^m M_{n_k}(D_k)$ is the Wedderburn decomposition ($m, n_k \in \mathbb{Z}_{\geq 1}$, $D_k$ rational division algebras), then for each $k$, \[ \Z(D_k) \simeq \QQ(\sqrt{-d})\] for some $d \in \mathbb{Z}_{\geq 0}$, where $\Z$ denotes the center.
 \item\label{prop:char} For each $\chi \in \Irr(G)$, $\QQ(\chi) = \mathbb{Q}(\sqrt{-d})$ for some $d \in \mathbb{Z}_{\geq 0}$, i.e.\ the field of character values of $\chi$ is $\QQ$ or an imaginary quadratic extension of $\QQ$ for each absolutely irreducible character of $G$.
\end{enumerate}
\end{proposition}

\begin{proof} The equivalence of \eqref{prop:cut}, \eqref{prop:conjugategenerators}, \eqref{prop:wedderburndecomp}, and \eqref{prop:char} is proved in \cite{RS}. Clearly, \eqref{prop:isr} implies \eqref{prop:conjugategenerators}. To see that \eqref{prop:conjugategenerators} $\Rightarrow$ \eqref{prop:isr}, let $x \in G$ and $j \in \mathbb{Z}$ with $(j, o(x)) = 1$ be given. We can choose $k = o(x)\cdot s + j$ for some $s\in \mathbb{Z}$ such that $(k, |G|) = 1$ (this is possible, for instance,  by Dirichlet's Theorem on the infinitude of primes in arithmetic progressions
). 
Then $x^j = x^k$ is conjugate to $x$ or $x^{-1}$, by \eqref{prop:conjugategenerators}. \end{proof}

We will use these equivalences freely without further mention. From \eqref{prop:wedderburndecomp} it easily follows that each quotient of a \cut\ group is again a \cut\ group. Note that $\mathbb{Q}(G)$, the field extension of the rationals generated by the values of all irreducible characters of the \cut\ group $G$ might have degree larger than $2$ over $\mathbb{Q}$ as can already be seen by the non-abelian group of order $21$. We have
\[ \{G \colon G \text{ is rational} \} \quad \subseteq \quad \{G \colon G \text{ is \cut} \} \quad \subseteq \quad \{G \colon G \text{ is semi-rational} \}.\]
In case of groups of odd order, the first class only contains the trivial group, and the last two classes coincide, cf.\ \cite[Remark 13]{CD}.

We will use the following notation introduced by Chillag and Dolfi together with their observations we state here.

\begin{definition} Let $G$ be a group, $x \in G$. Set $B_G(x) = \N_G(\langle x \rangle)/\C_G(\langle x \rangle)$, the automorphisms of $\langle x \rangle$ induced by elements of $G$. We identify $B_G(x)$ with the corresponding subgroup of $\Aut(\langle x \rangle)$. \end{definition}

\begin{lemma}[{\cite[Lemma 5 (2), (3)]{CD}}]\label{lemma_CD} Let $G$ be a group.
\begin{enumerate}
 \setcounter{enumi}{1}
 \item An element $x$ is inverse semi-rational in $G$ if and only if $\Aut(\langle x \rangle) = B_G(x)\langle \tau \rangle$, where $\tau \in \Aut(\langle x \rangle)$ is the inversion automorphism of $\langle x \rangle$ defined by $\tau(y) = y^{-1}$. In particular for an inverse semi-rational element $x$ in $G$, we have $[\Aut(\langle x \rangle) : B_G(x)] \leq 2$.
 \item If $p$ is a prime of the form $p \equiv 1 \bmod 4$, then a $p$-element $x$ is inverse semi-rational in $G$ if and only if $x$ is rational in $G$.
\end{enumerate}
\end{lemma}

\section{Solvable Groups}\label{sect:solvable}

It is a classical result of R.~Gow \cite{Gow} that the orders of solvable rational groups are divisible only by the primes $2$, $3$ and $5$. In such results, the treatment of the smallest primes usually requires the most effort; Gow used local Schur indices to show that $7$ cannot divide the order of a solvable rational group. The primes spectra of semi-rational groups were studied by Chillag and Dolfi \cite{CD}. They showed that $\pi(G) \subseteq \{2,3,5,7,13,17\}$ for a solvable semi-rational group $G$. It seems to be still unknown whether $17$ can divide the order of a solvable semi-rational group. For solvable groups $G$ that are inverse semi-rational (i.e.\ \cut), they prove $\pi(G) \subseteq \{2,3,5,7,13\}$, however it remained open whether $13$ can be omitted from the list. In \cite[Theorems 1 and 2]{Maheshwary} Maheshwary shows that $\pi(G) \subseteq \{2, 3, 5, 7\}$ for $G$ a (solvable) \cut\ groups of odd order or a solvable \cut\ groups of even order such that each element has order a prime power (with different techniques than we use here). To prove that this is best possible even for all solvable \cut\ groups, we improve the above mentioned result of Chillag and Dolfi \cite[Theorem 2]{CD} on solvable inverse semi-rational groups, and show that $13$ can be removed from the list.

In the investigation of the prime spectra of classes of solvable groups, the notion of the $k$-eigenvalue property, as introduced in \cite[(2.8) Definition]{FS} by E.~Farias e Soares, has proved to be very useful.

\begin{definition} Let a group $G$ act on a finite dimensional vectorspace $V$ over a finite field $F$. We say that $V$, or the action of $G$ on $V$, has the \emph{$k$-eigenvalue property}, for some $k \in \mathbb{Z}_{> 0}$ provided that (a) $k$ divides $|F^\times|$ and (b) for every $v \in V$, there exists $g \in G$ such that $v^g = \lambda v$, where $\lambda$ is some fixed element of order $k$ in $F^\times$. \end{definition}

A key point in the proof of Theorem \ref{ps_solv_cut} is the following result due to Farias e Soares, so that we will restate it here. By $\zeta_n$ we denote a primitive $n$th root of unity.

\begin{theorem}[{\cite[(2.9) Theorem B (c)]{FS}}]\label{TheoremFeS} Let $H$ be a solvable group acting on a finite-dimensional vector space $V$ over $\FF_p$, where $p \nmid |H|$. Assume that the action has the $k$-eigenvalue property. Let $\varphi$ be a Brauer character afforded by $V$ and $K = \QQ(\varphi)$. If $k = |K(\zeta_p):K|$ and $|K:\QQ| \leq 2$, then $p \in \{2, 3, 5, 7\}$.\end{theorem}

\begin{proof}[Proof of Theorem \ref{ps_solv_cut}] By \cite[Theorem 2]{CD} we have $\pi(G) \subseteq \{2, 3, 5, 7, 13 \}$ for solvable \cut\ groups $G$. Assume that $G$ is a solvable \cut\ group of minimal order such that $13 \in \pi(G)$. Let\[1 = G_{n+1} \unlhd G_{n} \unlhd ...  \unlhd G_{1} \unlhd G_0 = G\] be a  chief series of $G$. Note that each factor group $G/G_j$ is again a \cut\ group. Hence $G_n$ is the only chief factor of $G$ divisible by $13$ and, as a minimal normal subgroup, elementary abelian. Thus $V = G_n \in \Syl_{13}(G)$ and  $G = V \rtimes_\alpha H$, where $H \simeq G/V$ is also a \cut\ group of order not divisible by $13$. Now, the group homomorphism $\alpha \colon H \to \Aut(V) \simeq \GL(d, 13)$ turns $V$ into an $\FF_{13}H$-module, which is irreducible as $V$ is a minimal normal subgroup of $G$. We will write $V$ additively. Let $\varphi$ be the Brauer character of $H$ afforded by $V$. Then $\varphi$ is actually an ordinary character as $13 \nmid |H|$. By \cite[Lemma 3]{Tent} and the remark following it, the field of character values of $\varphi$ is contained in the field of character values of each of its irreducible constituents. But then $[\QQ(\varphi):\QQ] \leq 2$ by Proposition \ref{RitterSehgal} (1) $\Rightarrow$ (2). All elements of order $13$ in $G$ are rational in $G$ by Lemma \ref{lemma_CD}, so for $x \in V$ we have that $j\cdot x \sim x$ for each $j \in \mathbb{Z}$ with $(j, 13) = 1$, hence there exists $h \in H$ such that $x^h = j\cdot x$. In particular, the $H$-module $V$ has the $12$-eigenvalue property. But then we are in the situation of Theorem \ref{TheoremFeS}: $H$ is solvable, $K = \QQ(\varphi)$ has degree at most $2$ over $\QQ$ and $k = [K(\zeta_{13}): K] = 12$, as $K \cap \QQ(\zeta_{13}) = \QQ$ (the unique quadratic subfield of $\QQ(\zeta_{13})$ is real) and $H$ (and hence $G$) cannot exist. \end{proof}

\begin{remark} By looking a bit closer at the group structures, one can see that a set $S$ is the prime spectrum of a solvable \cut-group if and only if $S$ is one of the following sets: \begin{align*} \{2\},\ \{3\},\ \{2, 3\},\ \{2, 5\},\ \{3, 7\},\ \{2, 3, 5\},\ \{2, 3, 7\},\  \{2, 3, 5, 7\} . \end{align*}
By the bound on the prime spectrum in Theorem \ref{ps_solv_cut}, the description of \cut-groups of odd order in \cite[Theorem 3]{CD} and the examples in \cite[Theorem 5]{BMP} it only remains to consider $\{2, 7\}$, $\{2, 5, 7\}$ and $\{2, 3, 5, 7\}$. For the first two this can be done by purely arithmetic conditions: note that the generators of a cyclic subgroup of order $7$ decompose into one or two orbits under the conjugation action. But this means that the group has a section of order $6$ or $3$ which is impossible. The solvable \cut-groups $(C_5^2 \rtimes Q_8) \times (C_7 \rtimes C_3)$ and $S_3 \times (C_5^2 \rtimes Q_8) \times (C_7^2 \rtimes \SL(2,3))$ both have prime spectrum $\{2, 3, 5, 7\}$. \end{remark}

\section{Frobenius Groups}\label{sect:Frobenius}

Frobenius groups that are rational have been determined in \cite{DS} (there are two infinite series together with ``Markel's group''), and the structure of semi-rational Frobenius groups has been studied in \cite{ADD}, though a complete classification has not been obtained. We will now prove the classification of \cut\ Frobenius groups given in Theorem \ref{Frobenius_cut} in the introduction. To this end we will first show that there are only eight isomorphism types of potential Frobenius complements in \cut\ Frobenius groups (Proposition \ref{FrobeniusComplements}) and use this to determine the structure of all \cut\ Frobenius groups. The general structure of Frobenius groups is fairly well understood, standard references include \cite{Huppert1,Passman,ScottGT}.

In the following proposition we strengthen the conditions of \cite[Lemma~2.5]{ADD} for \cut\ groups.  Note that we can in particular prove that \cut\ Frobenius groups are solvable, which is not the case for semi-rational Frobenius groups (e.g.\ the Frobenius group $C_{11}^2 \rtimes \SL(2,5)$ is indeed semi-rational but not \cut). Small parts of the proofs are similar to those in \cite{ADD}, however we include the complete proofs to make the exposition self-contained.

\begin{proposition}\label{SylowSubgroupsComplement} Let $G$ be a \cut\ Frobenius group. Then $G$ is solvable. If $K$ is a Frobenius complement of $G$, then the Sylow $p$-subgroups of $K$ are cyclic of order at most $p$ if $p$ is odd and cyclic of order at most $4$ or quaternion of order $8$ for $p = 2$. \end{proposition}

\begin{proof} Assume by way of contradiction that the \cut\ Frobenius group $G$ is not solvable. Then a Frobenius complement $K$ of $G$ is not solvable and, by a theorem of Zassenhaus \cite[18.6.]{Passman}, $K$ has a normal subgroup of index at most $2$ isomorphic to $\SL(2, 5) \times M$, where $M$ is of order coprime to $30$. By \cite[Theorem~11.4.10]{JdR1} $G$ has a quotient isomorphic to $\SL(2, 5)$ or to the covering group of $S_5$ having SmallGroup ID \texttt{[240, 89]} in \textsf{GAP} \cite{GAP4}. It is readily checked that the last mentioned group has fields of character values isomorphic to $\QQ(\sqrt{2})$ and $\QQ(\sqrt{3})$ and $\SL(2, 5)$ has a character field $\QQ(\sqrt{5})$, so these groups cannot be \cut\ groups. As one of the two necessarily appears as a quotient of $G$, $G$ is also not \cut, implying that \cut\ Frobenius groups are necessarily solvable.

Let $p$ be an odd prime and $P \in \Syl_p(K)$ with $P \not= 1$. Then $P = \langle x \rangle \simeq C_{p^f}$ by \cite[18.1]{Passman} and $|\Aut(P)| = (p-1)p^{f-1}$. As $P$ is abelian, $p \nmid |B_K(x)|$. But then $[\Aut(P) : B_K(x)] \leq 2$ (see Lemma \ref{lemma_CD}) implies that $f \leq 1$.
 
For $p = 2$ the Sylow $2$-subgroups of $K$ are cyclic or generalized quaternion by \cite[18.1]{Passman}. If $P \in \Syl_2(K)$ is cyclic, then $K$ has a normal $2$-complement by Cayley's normal $2$-complement theorem, i.e.\ $K \simeq M \rtimes P$ with $M$ a group of odd order. Then $P$ is a cyclic $2$-group that is a \cut\ group as a quotient of a \cut\ group and hence of order at most $4$. Now assume that
\[ P = \langle\ x, y\ |\ x^{2^{n-1}} = 1,\ y^2 = x^{2^{n-2}},\ x^y = x^{-1}\ \rangle \]
is a (generalized) quaternion Sylow $2$-subgroup of $K$ of order $2^n$, $n \geq 3$. Then $\langle x \rangle \unlhd P$ is a cyclic normal subgroup of order $2^{n-1}$. Clearly $B_P(x) = \langle \tau \rangle$, where $\tau$ denotes the inversion-automorphism of $\langle x \rangle$. But as $P \in \Syl_2(K)$, the $2$-part of the order of $B_K(x)$ and $B_P(x)$ coincide. As $|\Aut(\langle x \rangle)| = \phi(2^{n-1}) = 2^{n-2}$, a power of $2$, we have  by Lemma \ref{lemma_CD}, that $x$ is inverse semi-rational in $K$ if and only if $2^{n-2} = 2$, hence $n = 3$ and $P$ is the quaternion group of order $8$.
\end{proof}

From the fact that the Frobenius complement $K$ is a solvable \cut\ group, we infer with Theorem \ref{ps_solv_cut} that $\pi(K) \subseteq \{2, 3, 5, 7\}$, so in particular there are only finitely many isomorphism types of Frobenius complements in Frobenius groups that are \cut.

\begin{proposition}\label{FrobeniusComplements} If $K$ is a Frobenius complement that is \cut, then $K$ is isomorphic to one of the following groups: \[ C_2,\quad C_3,\quad C_4,\quad C_6,\quad Q_8,\quad C_3 \rtimes C_4,\quad \SL(2, 3),\quad Q_8 \times C_3. \] 
\end{proposition}

\begin{proof} To filter the correct complements, we use $\pi(K) \subseteq \{2, 3, 5, 7\}$, the structure of the Sylow subgroups obtained in Proposition \ref{SylowSubgroupsComplement} together with the following additional results on Frobenius groups which eliminate another $7$ possibilities\footnote{Namely $S_3$, $C_5 \rtimes C_4$ with faithful action, $C_7 \rtimes C_3$, $C_7 \rtimes C_6$ with faithful action, $C_2 \times (C_7 \rtimes C_3)$, $C_{15} \rtimes C_4$ with faithful action and $C_2 \times (C_7 \rtimes C_3)$.}:
\begin{enumerate}
 \item If the Frobenius complement is of even order, then it has a unique (hence central) involution \cite[Theorem~18.1]{Passman}.
 \item A Frobenius complement cannot contain a subgroup which is itself a Frobenius group \cite[12.6.11]{ScottGT}.
\end{enumerate} 

This can be done by working out the possible structures of such groups or by a search in \textsf{GAP} of the groups meeting these requirements. By the bound of the order obtained by the structures of the Sylow subgroups this is a finite problem.
\end{proof}

In a large part of the proof of Theorem \ref{Frobenius_cut} we will have to deal with Frobenius groups $G$ that have an elementary abelian $p$-group as Frobenius kernel $F$. If, in this situation, $K$ denotes a Frobenius complement of $G$, then we can consider $F$ as an $\FF_p K$-module, which is semi-simple. Hence $F \simeq V_1 \oplus ... \oplus V_s$ for irreducible $\FF_p K$-modules $V_j$, which are necessarily faithful. Each $V_j$ then corresponds to an ordinary character of $K$. If $V$ is an $\FF_p K$-module, such that $V \rtimes K$ is a Frobenius group with complement $K$, then we can define a new group (a subdirect product) by taking $V^s \rtimes K$ with ``diagonal action'', i.e.\ $K$ acts on all $s$ copies of $V$ in the same way, that is if $(x_1, ..., x_s) \in V^s$, then $(x_1, ..., x_s)^y = (x_1^y, ..., x_s^y)$ for $y \in K$ (and $x_j^y$ is defined in $V \rtimes K$). Clearly, this is again a Frobenius group.


\begin{definition} Let $V$ be a $\FF_p K$-module such that $G = V \rtimes K$ is a Frobenius group. Let $\beta \colon K \to \GL_{\FF_p}(V)$ be the corresponding homomorphism. For $x \in V$ denote by $Z_G(x)$ those automorphisms of $\langle x \rangle$ that are induced by elements of $\beta(K) \cap \Z(\GL_{\FF_p}(V))$.
\end{definition}

Thus the elements of $Z_G(x)$ are exactly those automorphisms of $\langle x \rangle$ induced by those multiples of the identity of $V$ that come from elements of $K$. Actually, $Z_G(x)$ is equal for each $x \in V \setminus \{0\}$. Note that if in the previous definition the $\FF_p K$-module $V$ is absolutely irreducible, then $Z_G(x)$ are those automorphisms of $\langle x \rangle$ that are induced by elements of $\beta(\Z(K))$. For later use we record the following observation.

\begin{lemma}\label{copies_of_cut_groups} Let $K$ be a group, let $p$ an odd prime and let $V$ be an $\FF_pK$-module such that $G = V \rtimes K$ is a \cut\ Frobenius group. Then the following are equivalent:
\begin{enumerate}
 \item\label{lemma:Vs} $V^s \rtimes K$ (with diagonal action) for all $s \geq 2$ is a \cut\ Frobenius group.
 \item\label{lemma:V2} $V^2 \rtimes K$ (with diagonal action) is a \cut\ Frobenius group.
 \item\label{lemma:Z_G} $Z_G(x)\langle\tau \rangle = \Aut(\langle x \rangle)$ for each $x \in V$, where $\tau$ is the inversion automorphism of $\langle x \rangle$.
\end{enumerate}
\end{lemma}

\begin{proof} We will only give a proof in case $p \equiv 1 \bmod 4$, the proof in the other case is similar. Hence an element $x$ of order $p$ is inverse semi-rational in a group $X$ if and only if $B_X(x) = \Aut(\langle x \rangle)$, by Lemma \ref{lemma_CD}.

Clearly, \eqref{lemma:Vs} implies \eqref{lemma:V2}. Now assume that \eqref{lemma:V2} holds and we want to show \eqref{lemma:Z_G}.
Hence assume that $H = V^2 \rtimes K$ is a \cut\ Frobenius group. We have to show that all automorphisms of $\langle x \rangle$ for $x \in V \leq G$ are induced by scalar multiples of the identity that come from an element of $K$. Let $x \in V \setminus \{ 0 \}$ and $\lambda \in \FF_p^\times$. Then there is a unique $k \in K$ such that $x^k = \lambda x$, because $G$ is a \cut\ Frobenius group. Now for any $y \in V$, there is an $\ell \in K$ such that $(x, y) \in V^2 \leq H$ gets conjugate by $\ell$ to $\lambda (v, w)$, as $V^2 \rtimes K$ is a \cut\ group. Now as the action is diagonal, $(x^\ell, y^\ell) = (x, y)^\ell = \lambda(x, y) = (\lambda x, \lambda y) = (x^k, \lambda y)$. By uniqueness, $k = \ell$ and the action of $k$ multiplies all elements of $V$ by the same scalar $\lambda$, hence multiplication by $\lambda$ is an element of $Z_G(x)$ for every $x \in V$.

For the implication \eqref{lemma:Z_G} $\Rightarrow$ \eqref{lemma:Vs}, assume that $Z_G(x) = \Aut(\langle x \rangle)$ for each $x \in V$ and set $L = V^s \rtimes K$ for some $s \geq 2$ with diagonal action. Then $L$ is a Frobenius group. For each $y \in V^s$, $B_L(y) = \Aut(\langle y \rangle )$ and $y$ is (inverse semi-) rational in $L$. As $L$ is the union of conjugates of the Frobenius complement $K$ (which is \cut\ as a quotient of the \cut\ group $G$) and the Frobenius kernel $V^s$, it follows that $L$ is a \cut\ group. \end{proof}

\begin{proof}[Proof of Theorem \ref{Frobenius_cut}] Let $G$ always be a Frobenius group with Frobenius kernel $F$ and Frobenius complement $K$.
\begin{enumerate}
 \item If $K$ has even order, then the Frobenius kernel $F$ is abelian \cite[V, 8.18]{Huppert1}. Let $x \in F \setminus \{1\}$. Then $\C_G(x) = F$ and hence $B_G(x) = \N_G(\langle x \rangle)/\C_G(\langle x \rangle)$ is isomorphic to a subgroup of $K$. Note that if $x \in F$ is an element of order $r^n$ for an odd prime $r$, then we need that $\phi(r^n)/2 = \frac{(r-1)r^{n-1}}{2}$ divides $|K|$, by Lemma \ref{lemma_CD}; if in this situation the $2$-part of $(r-1)$ is at least $4$, then we even have that $\phi(r^n)$ divides $|K|$, as $\Aut(C_{r^n})$ is cyclic. From this it follows that $n \leq 1$, i.e.\ all Sylow subgroups of $F$ are elementary abelian.

 Going through the possibilities for even order Frobenius complements in \cut\ groups given in Proposition \ref{FrobeniusComplements}, we obtain the candidates for Frobenius kernels stated in Table \ref{table_pos_frob_kernels} sorted according to their Frobenius complements. 
 
 \begin{table}[h!]\caption{Possible isomorphism types of Frobenius complements $K$ together with their Frobenius kernels $F$ of \cut\ Frobenius groups.}
 \label{table_pos_frob_kernels}
 \begin{center}
\begin{tabular}{ll}
 \hline
  $K$ \hspace{2.2cm} & Candidates for $F$  \\ \hline\hline 
  $C_2$ & $C_3^b$  \\
  $C_4$ & $C_3^b$, \quad $C_5^c$, \quad $C_3^b \times C_5^c$ \\
  $C_6$ & $C_7^d$  \\
  $Q_8$ & $C_3^b$, \quad $C_5^c$, \quad $C_3^b \times C_5^c$  \\
  $C_3 \rtimes C_4$ & $C_5^c$, \quad $C_7^d$, \quad $C_5^c \times C_7^d$ \\
  $\SL(2,3)$ & $C_5^c$, \quad $C_7^d$, \quad $C_5^c \times C_7^d$ \\
  $Q_8 \times C_3$  & $C_5^c$, \quad $C_7^d$, \quad $C_5^c \times C_7^d$\\ \hline
\end{tabular}
\end{center}
\end{table}

 For each of these possibilities, we will now prove that either there is a unique \cut\ Frobenius group of that form or that no \cut\ Frobenius group of that form exists. We will first handle the cases where the Frobenius kernel $F$ is an elementary abelian $p$-group, as we have representation theory at our disposal in these situations. When considering $F$ as an $\FF_p K$-module, it will be written additively. For the convenience of the reader we provide the \textsf{GAP} SmallGroup ID of the group whenever this is appropriate.
 
 \begin{enumerate}[label=(\roman*), ,ref=\roman*, wide, labelwidth=!, labelindent=0pt]
 
   \item\label{C7dC6} $F = C_7^d$, $K = C_6$: We provide details in this case. We will prove that there is a unique \cut\ Frobenius group with kernel $F$ and complement $K$ for each $d$.  Note that $\FF_7$ is a splitting field for $C_6$ and there are two non-isomorphic faithful irreducible $\FF_7 C_6$-modules $W_1$ and $W_2$. A generator $t$ of $C_6$ acts as multiplication by $\bar{3}$ and $\bar{5}$ on $W_1 \simeq \FF_7$ and $W_2 \simeq \FF_7$, respectively. Clearly $W_1 \rtimes C_6$ and $W_2 \rtimes C_6$ are \cut\ Frobenius groups and so are  $W_1^d \rtimes C_6$ and $W_2^d \rtimes C_6$ for all $d \geq 1$ by Lemma \ref{copies_of_cut_groups}. We claim that if \begin{align}\label{F_as_F7C6-module} (V_1 \oplus ... \oplus V_d) \rtimes C_6, \end{align} where each $V_j$ is either isomorphic to $W_1$ or $W_2$, is a \cut\ Frobenius group, then all $V_j$'s have to be isomorphic. For if not, then for the element $x = \left(\begin{smallmatrix} \bar{1}, \bar{1}, ..., \bar{1}  \end{smallmatrix}\right) \in V_1 \oplus V_2 \oplus... \oplus V_d$, the subgroup $\langle x \rangle$ contains only two $G$-conjugates of $x$, i.e.\ $B_G(x) \simeq C_2$, and $x$ is not inverse semi-rational in $G$. Note that the isomorphism type of $G$ does not depend on whether all $V_j$'s in \eqref{F_as_F7C6-module} are isomorphic to $W_1$ or to $W_2$ as $t \mapsto t^5$ induces an automorphism of $C_6$, which in turn induces an isomorphism $W_1^d \rtimes C_6 \to W_2^d \rtimes C_6$. Thus, up to isomorphism, we get a unique \cut\ Frobenius group of the form $C_7^d \rtimes C_6$ for each $d \in \mathbb{Z}_{\geq 1}$ as stated in \eqref{theo:C7dC6} of Theorem \ref{Frobenius_cut}. 
 
   \item $F = C_5^c$, $K = C_4$:  With the same argument as in the previous case, we see that there is a unique \cut\ Frobenius group as given in \eqref{theo:C5cC4}.
 
   \item $F = C_3^b$, $K = C_2$: There is a unique faithful irreducible $\FF_3C_2$-module, where $C_2$ acts by inversion on $\FF_3$. 
   Giving us case \eqref{theo:C32bC2} of Theorem \ref{Frobenius_cut}.

   \item $F = C_3^b$, $K = C_4$: There is a unique faithful irreducible $\FF_3C_4$-module $W$ which is $2$-dimensional as $\FF_3$-vector space and for a generator $t$ of $C_4$, 
   $t^2$ acts as $\left(\begin{smallmatrix} \bar{2} & . \\ . & \bar{2} \end{smallmatrix}\right)$ on $W$. 
   Hence $W \rtimes C_4$ is a \cut\ Frobenius group and, by Lemma \ref{copies_of_cut_groups}, we get that $W^s \rtimes C_4$ is a \cut\ Frobenius group for each $s \geq 1$, which is case \eqref{theo:C32bC4}.

   \item $F = C_3^b$, $K = Q_8$: Again there is a unique faithful irreducible $\FF_3Q_8$-module $W$ which is $2$-dimensional as $\FF_3$-vectorspace and the unique central involution of $K$ acts as $\left(\begin{smallmatrix} \bar{2} & . \\ . & \bar{2} \end{smallmatrix}\right)$ on $F$. We can apply the same argument as above and see that there is a \cut\ Frobenius group of the form $C_3^b \rtimes Q_8$ if and only if $b$ is even. Thus we are in case \eqref{theo:C32bQ8}
   
   \item $F = C_5^c$, $K = Q_8$: Also in this case there is a unique faithful irreducible $\FF_5Q_8$-module $W$ which is $2$-dimensional as $\FF_5$-vectorspace. Now it can be checked that $W \rtimes Q_8$ is even a rational Frobenius group (also known as Markel's group, \textsf{GAP} SmallGroup ID \texttt{[200, 44]}). Note that only the center of $Q_8$ acts by diagonal matrices on $W$. Hence by Lemma \ref{copies_of_cut_groups}, $W^s \rtimes Q_8$ can never be a \cut\ Frobenius group if $s\geq 2$. Thus there is a \cut\ Frobenius group of the form $C_5^c \rtimes Q_8$ if and only if $c = 2$; this is case \eqref{theo:C52Q8}.
   
   \item $F = C_5^c$, $K = C_3 \rtimes C_4$: The group $K$ has a unique faithful irreducible character, which is of degree $2$ and rational valued. Hence there is a unique faithful $\FF_5K$-module $W$ which is $2$-dimensional as $\FF_5$-vectorspace. It can be checked that $G = W \rtimes K$ (\textsf{GAP} SmallGroup ID \texttt{[300, 23]}) is in fact a \cut\ Frobenius group. However $\Z(K) \simeq C_2$ and hence $W^s \rtimes K$ cannot be \cut\ for $s \geq 2$, by Lemma \ref{copies_of_cut_groups}, and we just obtain the group in \eqref{theo:C52C3C4}.
   
   \item $F = C_7^d$, $K = C_3 \rtimes C_4$:\label{C7dC3C4} As in the previous case there is a unique faithful $\FF_7K$-module $W$ which is $2$-dimensional as $\FF_7$-vector space. However in this case the group $W \rtimes K$ (\textsf{GAP} SmallGroup ID \texttt{[588, 33]}) is not \cut\ as for three of its irreducible characters the field of character values is the maximal real subfield of $\QQ(\zeta_7)$.
   
   \item $F = C_5^c$, $K = \SL(2,3)$: By looking at the ordinary character table, one sees that there are three faithful characters of $\SL(2,3)$, however only the representation corresponding to the rational valued character can be used to define a Frobenius group. 
   Denote the corresponding $\FF_5\SL(2,3)$-module by $W$. Then $W$ is absolutely irreducible and $\dim_{\FF_5} W = 2$. It is easily seen that $W \rtimes \SL(2,3)$ is a \cut\ group (e.g.\ because $\SL(2,3)$ acts transitively on $W\setminus \{0\}$). As $\Z(\SL(2,3))$ is cyclic of order $2$, there cannot be a \cut\ Frobenius group of the form $C_5^c \rtimes \SL(2,3)$ for $c > 2$, by Lemma \ref{copies_of_cut_groups}. Hence there is a unique \cut\ Frobenius group of the form $C_5^c \rtimes \SL(2,3)$ and $c = 2$ in this case (with \textsf{GAP} SmallGroup ID \texttt{[600, 150]}). Case \eqref{theo:C52SL23}.
   
   \item $F = C_7^d$, $K = \SL(2,3)$: With the same reasoning as in the previous case, there is a \cut\ Frobenius group of the form $C_7^d \rtimes \SL(2,3)$ if and only if $d = 2$. This group has the \textsf{GAP} SmallGroup ID \texttt{[1176, 215]}, this is case \eqref{theo:C72SL23}.
   
   \item\label{C5cQ8C3} $F = C_5^c$, $K = Q_8 \times C_3$: The group $K$ has two faithful irreducible ordinary characters, they are complex conjugates of each other and of degree $2$. Note that their field of values is $\QQ(\zeta_3)$, a corresponding representation in characteristic $5$ can be realized over $\FF_{25}$ and a corresponding $\FF_5K$-module $W$ has dimension $4$. The group $W \rtimes K$ has $20$ irreducible characters with field of character values $\QQ(\sqrt{5})$. Thus there is no \cut\ Frobenius group of the form $C_5^c \rtimes (Q_8 \times C_3)$.
   
   \item $F = C_7^d$, $K = Q_8 \times C_3$: From the previous case we see that there are two faithful irreducible $\FF_7K$-modules $W_1$ and $W_2$ which can be realized over $\FF_7$. Similar to in the case \eqref{C7dC6}, one can argue that in a \cut\ Frobenius group only one of the isomorphism types can be involved. Note that $\Z(K) \simeq C_6$ and hence these elements act in a faithful $\FF_7$-representation as scalar multiples of the identity. In particular $W_j \rtimes K$ is a \cut\ Frobenius group and then, by Lemma \ref{copies_of_cut_groups}, each $W_j^s \rtimes K$ is a \cut\ Frobenius group ($j \in \{1, 2\}$). Again as in the case \eqref{C7dC6} one can see that both, $W_1^s \rtimes K$ and $W_2^s \rtimes K$, lead to isomorphic groups. Hence there is a (unique) \cut\ Frobenius group of the form $C_7^d \rtimes (Q_8 \times C_3)$ if and only if $d$ is even, which is case \eqref{theo:C72dQ8C3}.
 \end{enumerate}
 
 It remains to handle the cases where the Frobenius kernel is divisible by two different primes. In what follows, we will always assume that $b, c, d$ are positive. Assume first that $G = (C_5^c \times C_7^d) \rtimes (C_3 \rtimes C_4)$ is a \cut\ Frobenius group with kernel $F = C_5^c \times C_7^d$. Then $G/O_5(G) \simeq C_7^d \rtimes (C_3 \rtimes C_4)$ is a \cut\ Frobenius group, contradicting case \eqref{C7dC3C4} above. Similarly, the existence of $(C_5^c \times C_7^d) \rtimes (Q_8 \times C_3)$ would contradict \eqref{C5cQ8C3}.

 \noindent \parbox[t]{\dimexpr\textwidth-\leftmargin}{%
   \vspace{-2.5mm}
   
   \hspace{8pt} Now assume that $G$ is a Frobenius group with kernel $C_3^b \times C_5^c$ and the 
    \begin{wrapfigure}[17]{r}{0.47\textwidth}
     \centering
     \vspace{-.5\baselineskip}
      \caption{Subgroups of $\Aut(C_{15}) \simeq C_2 \times C_4$}\label{AutC15}
     {\tiny
    \begin{tikzpicture}[node distance = 1.65cm, auto]
    \node (1) {1};
    \node (21)[above of=1, left of=1] {$\langle \alpha\rangle$};
    \node (22)[above of=1] {$\langle \beta^2\rangle$};
    \node (23)[above of=1, right of=1] {$\langle \alpha\beta^2\rangle$};
    \node (41)[above of=21] {$C_4 \simeq\langle \beta\rangle$};
    \node (42)[above of=22, align=center] {$\langle \alpha, \beta^2\rangle \simeq$ \\ $C_2 \times C_2$};
    \node (43)[above of=23] {$\langle \alpha\beta\rangle\simeq C_4$};
    \node (A)[above of=42] {$\langle \alpha, \beta\rangle \simeq C_2 \times C_4$};    
    \draw[-] (1) to node {} (21);
    \draw[-] (1) to node {} (22);
    \draw[-] (1) to node {} (23);
    \draw[-] (21) to node {} (42);
    \draw[-] (22) to node {} (41);
    \draw[-] (22) to node {} (42);
    \draw[-] (22) to node {} (43);
    \draw[-] (23) to node {} (42);
    \draw[-] (41) to node {} (A);
    \draw[-] (42) to node {} (A);
    \draw[-] (43) to node {} (A);
    \end{tikzpicture}
  }
\end{wrapfigure}
  complement $K$ is either $C_4$ or $Q_8$. Let $x \in G$ be an element of order $15$. Then $\Aut(\langle x \rangle) = \langle \alpha \rangle \times \langle \beta \rangle \simeq C_2 \times C_4$. As the Frobenius kernel is abelian, $\C_G(x) = F$ and we have that $B_G(x) = \N_G(\langle x \rangle) / \C_G(\langle x \rangle)$ is isomorphic to a subgroup of $K$. But it is also a subgroup of $\Aut(\langle x \rangle)$. As $x$ is inverse semi-rational, we have $|B_G(x)| \geq 4$, hence $B_G(x)$ is cyclic of order $4$ and $B_G(x) = \langle \bar{y} \rangle$, where $y$ acts on $x$ as $\beta$ or as $\alpha\beta$, cf.\ Figure \ref{AutC15}. In the first case $x^5$ is left invariant by $y$, hence $y \in \C_G(x^5)$, contradicting the fact that $G$ is Frobenius. In the second case $x^5$ is stabilized by $y^2$ and $y^2 \in \C_G(x^5)$, again a contradiction with $G$ being a Frobenius group.
 }

  The last remaining case is if $G$ is a \cut\ Frobenius group with kernel $F = C_5^c \times C_7^d$ and complement $K = \SL(2,3)$. If such a group existed, we could consider an element $x \in F$ of order $35$. Then $B_G(x) \leq \Aut(C_{35})$ has to have order $12$ or $24 = \phi(35)$. But $\SL(2,3)$ does not contain a subgroup of order $12$ and $\Aut(C_{35})$ is clearly not isomorphic to $\SL(2,3)$.
 
 \item If $|K|$ is odd, then $K \simeq C_3$ by Proposition \ref{FrobeniusComplements}, and the Frobenius kernel $F$ of $G$ is a $\{2, 5, 7\}$-group by Theorem \ref{ps_solv_cut}. If $|F|$ is be divisible by $5$, then the nilpotent Frobenius kernel $F$ contains an element $z$ of order $5$ in its center. Then $z$ has to be rational by Lemma \ref{lemma_CD}. But this is impossible as $B_G(z) \leq C_3$. We infer that $F$ is a $\{2, 7\}$-group. First assume that $|F|$ is divisible by both primes, $2$ and $7$. Then $F$ has an element $z$ of order $14$ in its center and $F \leq \C_G(x)$. As $G$ is a \cut\ group, $B_G(x) \simeq C_3$. But then the unique involution in $\langle x \rangle$ is central in $G$, which is a contradiction. Consequently, $F$ is a $p$-group for $p = 2$ or $p = 7$. By a classical result of Burnside, see e.g.\ \cite[V, 8.8]{Huppert1}, the Frobenius kernel $F$ is nilpotent of class at most $2$ and thus metabelian. Let $x \in \Z(F)$. Then $\C_G(x) = F$ and $B_G(x) \leq K \simeq C_3$. Hence $\phi(o(x))/2 \mid 3$ and thus $o(x) \mid 4$ (in case $p = 2$) or $o(x) \mid 7$ (in case $p = 7$). Now $F/F' \rtimes K$ is an epimorphic image of $G$, and therefore a \cut\ group with $F/F'$ being abelian. Thus for each $y \in F/F'$ we have $\phi(o(y))/2 \mid 3$ and $F/F'$ is of exponent $4$ (if $p = 2$) or of exponent $7$ (if $p = 7$), respectively. As $F$ is metabelian, $F' \leq \Z(F)$, and the claim on the structure of the group follows. 
 
 Assume now that $F$ is a $2$-group and $x \in F$ is a $2$-element. Then the order of $\Aut(\langle x \rangle)$ is a power of $2$ and hence $B_F(x) = B_G(x)$ and $F$ has to be \cut. The order of $F$ has to be an even power of $2$ as $3 \mid (|F| - 1)$. On the other hand, every \cut\ $2$-group admitting a fixed point free automorphism of order $3$ can clearly be used to define a \cut\ Frobenius group.
 
 If $F$ is a $7$-group, then clearly $\exp F$ is a divisor of $7^2$, by the above, we want to show that $\exp F = 7$. Deny the latter and assume that $\exp F = 49$, then there is an element $y \in F$ of order $49$. Then necessarily $B_G(y) \simeq C_{21}$, as $G$ is \cut. Hence there is an element $x \in F$ of order $k \in \{7, 49\}$ normalizing $\langle y \rangle$, but not centralizing it. But then $H = \langle x, t \rangle \simeq C_k \rtimes C_3$, where $t$ is a generator of the Frobenius complement, and $B_G(y)$ is an epimorphic image of $H$. But this is a contradiction since $H$ does not have an epimorphic image that is cyclic of order $21$. Hence $F$ has exponent $7$. Each $7$-element $x$ has to be inverse semi-rational in $G$, thus it has to have three conjugates in the cyclic subgroup generated by it, so the generator $t$ of a Frobenius complement has to normalize $\langle x \rangle$. Conversely, assume that $F$ is a $7$-group of exponent $7$ such that it admits a fixed point free automorphism $\alpha$ of order $3$ mapping each cyclic subgroup to itself. Then, clearly, $F \rtimes \langle \alpha \rangle$ defines a \cut\ Frobenius group. The result is proved. \qedhere
\end{enumerate}
\end{proof}

\begin{remark} As by \cite[Remark~13]{CD}, groups of odd order are semi-rational if and only if they are \cut\ (= inverse semi-rational), \eqref{kernel7group} of the Theorem \ref{Frobenius_cut} also gives a more precise description of the Frobenius groups appearing in \cite[Theorem 3]{CD}, where the semi-rational groups of odd order are classified. \newline
The arguments from the last part of the proof of \eqref{theorem_even_complement} can also be used to reduce the list of potential semi-rational Frobenius groups in \cite[Theorem~1.2]{ADD}, e.g.\ groups with complement $C_4$ or $Q_8$, and Frobenius kernels with an order divisible by two different primes can be removed from the list as well as groups with complement $Q_{16}$ and kernel being an elementary abelian $17$-group. Some of the arguments also apply to show that the isomorphism type of semi-rational Frobenius groups of a certain type are unique. In general, a semi-rational Frobenius group is not uniquely determined by its Frobenius kernel and complement. The two non-isomorphic Frobenius groups $(C_5 \times C_5) \rtimes C_4$ are both semi-rational (however, only one of them is \cut).
\end{remark}

 We now describe the \cut\ Frobenius groups with abelian kernel.
 
 \begin{corollary}\label{cor:abelian_kernel} Let $G$ be a \cut\ Frobenius group with abelian kernel. Then $G$ appears in Theorem \ref{Frobenius_cut} \eqref{theorem_even_complement} or is isomorphic to
   \[(C_2^a \times C_4^{a'}) \rtimes C_3 \quad \text{or} \quad C_7^d \rtimes C_3\qquad  (a, a' \in 2\mathbb{Z}_{\geq 0},\ a+a' > 0, \ d \in \mathbb{Z}_{\geq 1}).\]
   For each of the above structures there is a unique \cut\ Frobenius group of that form.
 \end{corollary}
 
 \begin{proof} If the order of the complement is even, then the group was handled in Theorem \ref{Frobenius_cut} \eqref{theorem_even_complement}. In the case of a complement of odd order, it follows from the proof of Theorem \ref{Frobenius_cut} \eqref{odd_complement} that the Frobenius kernel $F$ is an abelian group of exponent $7$ or a divisor of $4$. 
 
 If $\exp F = 7$, then we can argue as in \eqref{C7dC6} in the proof of Theorem \ref{Frobenius_cut}, to see that there is a unique \cut\ Frobenius group of the form and $C_7^d \rtimes C_3$ for all $d \in \mathbb{Z}_{\geq 1}$.
 
 Assume that $F$ is a $2$-group. As all groups of the form $C_2^a \times C_4^{a'}$ are \cut, the task amounts to find all Frobenius groups of the form  $(C_2^a \times C_4^{a'}) \rtimes C_3$. But it is well-known that there is such a group if and only if $a, a' \in 2\mathbb{Z}_{\geq 0}$, not both zero, and this group is unique up to isomorphism. This can be extracted, for example, from the more general discussion in \cite[Section~11]{BH}.
 \end{proof}

\begin{remark} Turning our attention to \cut\ Frobenius groups with non-abelian Frobenius kernels we provide examples in both situations, when the kernel is a $7$-group and when it is a $2$-group. For the first, we refer to \cite[Example~1]{CD}. There it is shown that the Heisenberg group $P$ of order $7^3$ (i.e.\ the non-abelian group of order $7^3$ and exponent $7$) admits a fixed-point free automorphism of order $3$, which results in a Frobenius group $P \rtimes C_3$ that is semi-rational (recall that in the case of groups of odd order, \cut\ and semi-rational are the same \cite[Remark~13]{CD}).

The group \begin{align*} X = \langle\ s, t, u, v \ | & \ s^4 = t^4 = u^2 = v^2 = 1 = [s, t] = [u, v],\\ &\ s^u = st^2,\ t^u = s^2t,\ s^v = s^3,\ t^v = s^2t^3 \ \rangle  \simeq (C_4 \times C_4) \rtimes (C_2 \times C_2 ) \end{align*} has exponent $4$ and admits a fixed-point free automorphism $\alpha$ of order $3$, given by
\[s^\alpha = s^3t^3, \quad t^\alpha = s,\quad u^\alpha = uv, \quad v^\alpha = u. \]
Hence $G = X \rtimes \langle \alpha \rangle$ is a Frobenius group. As $X$ has exponent $4$, $G$ is actually a \cut\ group. Using this group, we can construct \cut\ Frobenius groups with non-abelian Frobenius kernels of order $2^{6 + 2e}\cdot 3$, for each $e \in \mathbb{Z}_{\geq 0}$. 

Now consider the group \begin{align*} Y = \langle\ s, t, u, v \ | & \ s^4 = u^4 = 1,\ s^2 = t^2 \in \Z(Y),\ u^2 = v^2 \in \Z(Y), \\ &\ [s, t] = u^2, [s, u] =  s^2, [s, v] = u^2, \\ &\ [t, u] = u^2, [t, v] = s^2, [u, v] = s^2 \ \rangle.   \end{align*} Then $\Z(Y) = Y' = \langle s^2, u^2 \rangle$ is isomorphic to a Klein four group $C_2^2$ and $Y/\Z(Y)$ is isomorphic to $C_2^4$. The group $Y$ also admits a fixed-point free automorphism $\beta$ of order $3$ and $\exp Y = 4$, so that we get another example of a \cut\ Frobenius group $Y \rtimes \langle \beta \rangle$, now with a Frobenius kernel that is not split. The \textsf{GAP} SmallGroup IDs of these groups are \texttt{[64, 242]}, \texttt{[192, 1023]}, \texttt{[64, 245]} and \texttt{[192, 1025]}, respectively.

All other examples of \cut\ Frobenius groups with indecomposable (i.e.\ it cannot be written as direct product of smaller groups) non-abelian Frobenius kernel of order at most $2^8\cdot 3$ have SmallGroup IDs \texttt{[768, 1083600]}, \texttt{[768, 1083604]}, \texttt{[768, 1083733]}, \texttt{[768, 1085039]} with Frobenius kernel \texttt{[256, 6815]}, \texttt{[256, 6817]}, \texttt{[256, 10326]}, \texttt{[256, 55682]}, respectively. The latter all have exponent $4$ (except \texttt{[256, 10326]}, which has exponent $8$).
\end{remark}

A group $G$ is called a \emph{Camina group}, if its derived subgroup $G'$ is different from $1$ and $G$ and $x^G = xG'$ for all $x \in G \setminus G'$, i.e.\ the conjugacy class and the coset modulo $G'$ coincide for elements outside $G'$. (Note that in some definitions $G' = 1$ is allowed for Camina groups.) In \cite[Theorem~4]{BMP} it is proved that a Camina $p$-group has the \cut\ property if and only if $p \in \{2, 3\}$. With the classification of all Camina groups by R.~Dark and C.M.~Scoppla cf.\ \cite{DarkScoppola,IsaacsLewis} and the above classification of \cut\ Frobenius groups we obtain the following result. 

\begin{corollary}\label{cor:camina}
A Camina group $G$ is a \cut\ group if and only if one of the following holds:
\begin{itemize}
 \item $G$ is a $2$- or a $3$-group,
 \item $G$ is a Frobenius group of the form \[ (C_2^{2n} \times C_4^{2m}) \rtimes C_3,\ C_3^n \rtimes C_2,\ C_3^{2n} \rtimes C_4,\ C_3^{2n} \rtimes Q_8, \ C_5^{2} \rtimes Q_8 \] for $n, m \in \mathbb{Z}_{\geq 1}$.
 \item $G$ a Frobenius group of the form \[C_7^n \rtimes C_3,\ C_5^{n} \rtimes C_4,\ C_7^{n} \rtimes C_6, \] for $n \in \mathbb{Z}_{\geq 1}$, where a generator of the complement raises each element of the Frobenius kernel to the same power.
 \item $G$ is a \cut\ Frobenius group with a cyclic complement of order $3$ and non-abelian kernel as described in Theorem \ref{Frobenius_cut} \eqref{odd_complement}.
\end{itemize}
\end{corollary}

\noindent \textbf{Acknowledegment.}\quad The author is grateful to Silvio Dolfi for an interesting communication on inverse semi-rational groups and would like to thank Eric Jespers for stimulating discussions.

\bibliographystyle{amsalpha}
\bibliography{cut}

\end{document}